\documentclass[article, 11pt]{amsart}

\usepackage{amsmath, amssymb, amsfonts, amsthm, latexsym}
\usepackage{tikz}
\usepackage{algorithm, caption}
\usepackage[noend]{algpseudocode}
\usepackage{hyperref}

\newtheorem{Theorem}{Theorem}[section]
\newtheorem{Lemma}[Theorem]{Lemma}

\newtheorem{Proposition}[Theorem]{Proposition}
\newtheorem{Remark}[Theorem]{Remark}

\def\reg{\operatorname{reg}}

\def\ini{\operatorname{ini}}
\def\depth{\operatorname{depth}}

\def\ini{{\operatorname{in}}}

\def\mm{{\mathfrak m}}

\def\NN{{\mathbb N}}
\def\a{{\mathbf a}}

\def\e{{\mathbf e}}

\newcommand{\Break}{\State \textbf{break} }
\newcommand{\ra}{\rangle}
\newcommand{\la}{\langle}
\begin{document}

\title{On the computation of Castelnuovo-Mumford regularity of the Rees algebra and of the fiber ring}

\author{Dinh Thanh Trung}
\address{Department of Mathematics, FPT University, Hoa Lac Hi-Tech Park, Km29 Thang Long Blvd, Hanoi, Vietnam}
\email{trung.dinh.nb@gmail.com}
\thanks{The author is supported by Vietnam National Foundation for Science and Technology Development. He would also like to thank Vietnam Institute for Advanced Study in Mathematics for its  support and hospitality.}
\begin{abstract}
We present algorithms for the computation of the Castelnuovo-Mumford regularity of the Rees algebra and of the fiber ring of equigenerated $\mm$-primary ideals in two variables. Applying these algorithms, we find a counter-example to a conjecture of Eisenbud and Ulrich which states that these regularities are equal. 
\end{abstract}
\maketitle
\section*{Introduction}
Let $R$ be a standard graded algebra over a commutative ring with unity, and $R_+$ the graded ideal of elements of positive degrees.
Let $M$ be a finitely generated graded $R$-module.
Let $H_{R_+}^i(M)$ denote the $i$-th local cohomology module of $M$ with respect to $R_+$, and set $a_i(M) =  \max\{n|\ H_{R_+}^i(M)_n \neq 0\}$ with the convention $a_i(M) = -\infty$ if $H_{R_+}^i(M) = 0$. The Castelnuovo-Mumford regularity is defined by 
$$\reg M := \max\{a_i(M)+i|\ i \ge 0\}.$$
\par

It is well known that $\reg M$ controls the complexity of the graded structure of $M$ (see e.g.\cite{Tr0}). If $R(I) = \oplus_{n\ge 0} I^n$ is the Rees algebra of an ideal $I$ in a local ring, then $\reg R(I)$ is an upper bound for several invariants of $I$ such as the relation type, the reduction number, the postulation number, etc.

If $I$ is a graded ideal in a standard graded algebra $A$ over a field, Cutkosky, Herzog, and Trung \cite{CHT}, Kodiyalam \cite{Ko}, Trung and Wang \cite{TW} showed that $\reg I^n$ is asymptotically a linear function. However, the stability index of $\reg I^n$, that is, the least integer $n$ for which  $\reg I^n$ becomes a linear function afterwards, remains mysterious. If $I$ is equigenerated, i.e. I is generated by homogeneous elements of the same degree, Eisenbud and Ulrich \cite{EU} showed that  the stability of the function $\reg I^n$  is related to the presentation of the Rees algebra of $R(I)$ as a direct sum of modules over the fiber ring $F(I) = \oplus_{n\ge 0}I^n/\mm I^n$, where $\mm$ denotes the maximal graded ideal of $A$.  Inspired by this finding, they raised the following 

\noindent {\bf Conjecture}.  
{\it Let $A$ be a polynomial ring over a field $k$. Let $\mm$ be the maximal graded ideal of $A$ and $I$ an equigenerated homogeneous $\mm$-primary ideal.  Then $\reg R(I) = \reg F(I)$. }

The conjecture was originally formulated for a standard graded algebra over a field, however it was pointed out in \cite{RTT} that, if the conjecture were true then the base ring must be Buchsbaum. In fact, Ulrich has informed the authors of \cite{RTT} that actually,  the conjecture is meant for a polynomial ring. We note that, as it was observed in \cite{EU}, the inequality $\reg R(I) \geq \reg F(I)$ is always true if the base ring $A$ is a standard graded algebra over a field.  

It was shown in \cite{RTT} that there is a strong connection between Castelnuovo-Mumford regularity and Ratliff-Rush closure of $I$, which is defined as the ideal $\widetilde{I}=\bigcup_{n\geq 1} I^{n+1}: I^n$. In particular, $\reg R(I)$ and $\reg F(I)$ can be characterized in term of the behavior of the Ratliff-Rush filtration $\{\widetilde{I^n}\}$, when $A$ is a polynomial ring in two variables. In that paper the conjecture was confirmed for monomial ideals in $k[x,y]$ that either are 3-generated, or contain the monomials $(x^d,y^d,xy^{d-1})$. If $A$ is a standard graded Buchsbaum algebra over a field then the conjecture was also proved to be true when the depth of the associated graded ring of $I$ is at least $\dim A-1$. 

The main aim of this paper is to work out algorithms for the computation of $\reg R(I)$ and $\reg F(I)$ in the case $I$ is an equigenerated homogeneous ideal in two variables. The algorithms depend essentially on the computation of the Ratliff-Rush filtration. 

In the first section of this paper we recall the mentioned results of \cite{RTT} on connection between regularity and  Ratliff-Rush closure. We will also discuss the computation of Ratliff-Rush closures of  equigenerated homogeneous ideals in terms of the multiplicity. For large multiplicity this computation is not efficient however.

In the second section, we give other algorithms which work only for equigenerated monomial ideals in two variables. These algorithms are based on a method of Quinonez \cite{Cr} for the computation of Ratliff-Rush closures of such monomial ideals. We ran these algorithms to compare $\reg R(I)$ and $\reg F(I)$ for monomial ideals of the form $(x^d,x^ay^{d-a},x^by^{d-b},y^d)$, and we found the case $d=157, a=35, b=98$ is a counter-example to the conjecture of Eisenbud and Ulrich. That is the example of smallest degree among ideals of the above form. 

In the last section of the paper, we will give a sufficient condition for $\reg R(I)> \reg F(I)$ when $I$ is an equigenerated monomial ideal in two variables. This condition can be used to give a theoretical proof for the above counter-example to the conjecture of Eisenbud and Ulrich.

\section{Regularity and Ratliff-Rush filtration}

For a Noetherian ring $A$, the Ratliff-Rush closure of an ideal $I$ is defined as
$$\widetilde{I}=\bigcup_{n\geq 1} I^{n+1}:I^n.
$$
If $I$ is a regular ideal, then $\widetilde{I^n}=I^n$ for $n\gg 0$ \cite{RR}. Define $s^*(I)$ to be the least integer $m\geq 1$ such that $\widetilde{I^n}=I^n$ for all $n\geq m$.

An ideal $J\subseteq I$ is called a reduction of $I$ if there exists an integer $n$ such that $I^{n+1}=JI^n$. The least number $n$ with that property is called the reduction number of $I$ with respect to $J$, and is denoted by $r_J(I)$. A reduction is minimal if it is minimal with respect to containment. We state the following result of \cite{RTT} regarding the computation of $\reg R(I)$.

\begin{Theorem}\label{RegRees1}\cite[Theorem 2.4]{RTT}
Let $(A,\mm)$ be a two-dimensional Buchsbaum local ring with $\depth A>0$. Let $I$ be an $\mm$-primary ideal which is not a parameter ideal. Let $J$ be any minimal reduction of $I$. Then
$$
\reg R(I)=\max\{r_J(I), s^*(I)\}=\min \{n\geq r_J(I)|\ \widetilde{I^n}=I^n\}.
$$
\end{Theorem}

Let $I$ be a homogeneous ideal in a standard graded algebra over a field. Assume that $I$ is generated in degree $d$. We denote $s^*_{\ini}(I)$ to be the least integer $m\geq 1$ such that for all $n\geq m$ the ideals $\widetilde{I^n}$ and $I^n$ agree at the initial degree, that is $\big(\widetilde{I^n}\big)_{nd}=(I^n)_{nd}$ for all $n\geq m$. 

\begin{Theorem}\label{RegFiber1}\cite[Theorem 3.8]{RTT}
Let $(A,\mm)$ be a two-dimensional standard graded algebra over a field which is a Buchsbaum ring with $\depth A>0$. Let $I$ be an $\mm$-primary ideal generated in degree $d$. Assume that $I$ is not a parameter ideal. Let $J$ be any homogeneous minimal reduction of $I$. Then
$$
\reg F(I)=\max\{r_J(I), s^*_\ini (I)\}=\min \{n\geq r_J(I)|\ \big(\widetilde{I^n}\big)_{nd}=(I^n)_{nd}\}.
$$
\end{Theorem}

In order to use these results to compare regularity of the Rees algebra and of the fiber ring of an ideal $I$, it is essential to have an effective tool to compute Ratliff-Rush closure. The computation of Ratliff-Rush closure of an ideal $I$ in general is a hard problem, because $I^{n+1}:I^n=I^{n}:I^{n-1}$ does not imply that $I^{n+2}:I^{n+1}=I^{n+1}:I^{n}$ \cite {El, RS}. A bound for the least number $c$ such that $\widetilde{I}=I^{c}:I^{c-1}$ was given in \cite{RTT}; let us recall it here. For a positive integer $n$ the Ratliff-Rush closure of $I^n$ can be computed as
$$
\widetilde{I^n}=\bigcup_{t\geq n} I^{t}:I^{t-n}.
$$
By \cite[Proposition 2.1]{RTT}, the least integer $t$ such that $\widetilde{I^n}=I^{t}:I^{t-n}$ is  bounded above by $\reg R(I)$. Since $\reg R(I)$ is bounded by other invariants of $I$, this result provides a practical tool to compute Ratliff-Rush closure of powers of $I$. If $I$ is an $\mm$-primary ideal in a local Cohen-Macaulay ring of dimension $2$, by work in \cite{Li1,RTV1} we have $\reg R(I)\leq e(I)(e(I)-1)$, where $e(I)$ is the multiplicity of $I$. Therefore, in particular if $I$ is an $(x,y)$-primary ideal in $k[x,y]$ generated in degree $d$ then, since $e(I)=d^2$, we have $\widetilde{I^n}=I^r:I^{r-n}$, where $r=d^2(d^2-1)$. However for large $d$ this computation of Ratliff-Rush closure is not very efficient.

\begin{Remark} {\rm Macaulay2 is able to compute a minimal reduction of an ideal, using a probabilistic algorithm, and the reduction number with respect to this minimal reduction \cite{M2}. }
\end{Remark}
\section{The case of monomial ideals}

In this section we will focus on the class of monomial ideals in two variables. We will present efficient algorithms to compare $\reg R(I)$ and $\reg F(I)$, which are based on a method for the computation of Ratliff-Rush closures due to Quinonez \cite{Cr}.

Let $I$ be an $\mm$-primary monomial ideal in $k[x,y]$ generated in degree $d$. Then $I=(x^d,x^{a_1}y^{d-a_1},\ldots,x^{a_p}y^{d-a_p}, y^d)$, in which $a_1,a_2,\ldots,a_p$ is an increasing sequence of positive integers not exceeding $d-1$. 

We denote by $E$ the exponent vectors of the minimal generators of $I$, that is, $E=\{(0,d),(a_1,d-a_1),\ldots,(a_p,d-a_p),(d,0)\}\subseteq\NN^2$. We also define 
$$
S=\{0,a_1,\ldots,a_p,d\},
$$ 
and denote by $\la S\ra$ the numerical semigroup generated by the elements of $S$, that is, $\la S\ra =\{t_1 a_1+\cdots+t_p a_p+td|\ t_1,\ldots,t_p,t \text { are non-negative integers}\}$. Similarly we define
$$
T=\{d,d-a_1,\ldots, d-a_p,0\},
$$ 
and $\la T\ra$ the numerical semigroup generated by the elements of $T$. 

Throughout this paper, if $A$ and $B$ are sets in $\NN^t$ we use the notation $A+B$ for the Minkowski sum of $A$ and $B$, and for a positive integer $n$, the set $nA$ is the Minkowski sum of $n$ coppies of $A$. For a vector $\a\in\NN^2$, the notation $A+\a$ stands for the set $A+\{\a\}$.

Quinonez \cite[Proposition 3.7]{Cr} proved the formula 
$$
\widetilde{I}=(x^uy^{d-u}|\ u\in \la S\ra,\ u\leq d)\cap (x^{d-v}y^v|\ v\in \la T\ra,\ v\leq d).
$$
Using this formula we are led to the following description of $\widetilde{I^n}$. 

\begin{Theorem}\label{RR} Let $I$ be a monomial ideal in $k[x,y]$ generated in degree $d$, and assume $I$ contains $x^d,y^d$. Then $\widetilde{I^n}$ is generated by all monomials of the forms $x^uy^v$ where $u,v\leq nd\leq u+v,$ and either $u\in\la S\ra, v\in \la  T\ra$ or $nd-u\in\la T\ra,nd-v\in\la S\ra.$
\end{Theorem}

\begin{proof} 
The assertion can be deduced from \cite[Proposition 3.7]{Cr}. However we present here an alternative proof.

Let $J=(x^d,y^d)$. It is easy to check that $J$ is a minimal reduction of $I$. By \cite{Sa}, we have
$$
\widetilde{I^n}=\bigcup_{t\geq 1} I^{n+t}:(x^{td},y^{td}).
$$
This implies that
$$
\widetilde{I^n}=(x^ry^s|\ (r,s)+t\e_1,(r,s)+t\e_2\in (n+t)E+\NN^2 \text{ for some } t\geq 0),
$$
in which $\e_1=(d,0), \e_2=(0,d).$

Let $u,v$ be such that $u,v\leq nd\leq u+v$. If $u\in\la S\ra, v\in\la T\ra$, there is a sufficiently large $t$ such that $(nd-v,v)+t\e_1\in (n+t)E$ and $(u,nd-u)+t\e_2\in (n+t)E$. Similarly, if $nd-u\in\la T\ra, nd-v\in\la S\ra$, then there is some $t\geq 0$ such that $(u,nd-u)+t\e_1\in (n+t)E$ and $(nd-v,v)+t\e_2\in (n+t)E$. In either case, $(u,v)+t\e_1,(u,v)+t\e_2\in (n+t)E+\NN^2$. 

Conversely, let $(r,s)$ be a vector such that $(r,s)+t\e_1,(r,s)+t\e_2\in (n+t)E+\NN^2$. We shall see that $x^ry^s$ is divisible by a monomial $x^uy^v$ with $u,v\leq nd\leq u+v,$ and either $u\in\la S\ra, v\in \la T\ra$ or $nd-u\in\la T\ra, nd-v\in\la S\ra$. We may assume that $r,s\leq nd$. There are $(u_1,v_1), (u_2,v_2)\in (n+t)E$ such that $(r+td,s)-(u_1,v_1)$ and $(r,s+td)-(u_2,v_2)$ are vectors in $\NN^2$. It is then easy to see that $(r,s)-(u_2,v_1)$ and $(r,s)-(nd-v_1,nd-u_2)$ are also in $\NN^2$, which means that the monomial $x^ry^s$ is divisible by both $x^{u_2}y^{v_1}$ and $x^{nd-v_1}y^{nd-u_2}$. Note that $u_2\in\la S\ra$ and $v_1\in\la T\ra$.
\end{proof}

By Theorem \ref{RR}, to compute the Ratliff-Rush closure of $I^n$ we first need to find all pairs $(u, v)$ inside the square $[1,nd]\times [1,nd]$ such that $u\in \la S\ra$, $v\in \la T\ra$. Each such pair corresponds to a generator of $\widetilde{I^n}$, which is $x^uy^v$ if $u+v\geq nd$, or is $x^{nd-v}y^{nd-u}$ if $u+v\leq nd$.  

\begin{algorithm}
\caption{Computing $\widetilde{I^n}$ for $I=(x^d,x^{a_1}y^{d-a_1},\ldots,x^{a_p}y^{d-a_p},y^d)$}
\label{algRees}
\begin{algorithmic}[1]
\Procedure{RR}{$I,n$}
   \State $S\gets \{0,a_1,\ldots,a_p,d\}$
   \State $T\gets \{0,d-a_p,\ldots,d-a_1,d\}$
   \State $Gen\gets\emptyset$
   \For{$(u,v)$ in $\{1,\ldots,nd\}\times\{1,\ldots,nd\}$}
   			\If{$(u\in \la S\ra)\wedge(v\in\la T\ra)$}
   				\If{$u+v\geq nd$}
   					\State $Gen\gets Gen\cup\{(u,v)\}$
   				\Else
   					\State $Gen\gets Gen\cup\{(nd-v,nd-u)\}$
   				\EndIf
   			\EndIf
   		\EndFor
   	\State\textbf{return} $Gen$\Comment{\it{Exponent vectors of non-minimal generators of $\widetilde{I^n}$}}
\EndProcedure
\end{algorithmic}
\end{algorithm}

For the computation of $\reg R(I)$, we repeatedly check if $\widetilde{I^n}=I^n$ starting from $n=r_J(I)$ on. To verify if $\widetilde{I^n}=I^n$ we check if each monomial generator $x^uy^v$ in $\widetilde{I^n}$ is divisible by some generator $x^wy^{nd-w}$ in $I^n$. The monomial $x^wy^{nd-w}$ is in $I^n$ if and only if  the vector $(w,nd-w)$ is in $nE$, and if and only if $w$ is in $nS$. Thus we need to check if there is any integer $w$ between $nd-v$ and $u$, inclusive, such that $w\in nS$. 

By Theorem \ref{RegFiber1}, the regularity of the fiber ring $\reg F(I)$ is the least integer $n\geq r_J(I)$ such that $\big(\widetilde{I^n}\big)_{nd}=(I^n)_{nd}$. From Theorem~\ref{RR}, the set $\big(\widetilde{I^n}\big)_{nd}$ consists of the monomials $x^uy^v$ with $u+v=nd$ and $u\in \la S\ra, v\in\la T\ra$. Such a monomial belong to $(I^n)_{nd}$ if and only if $u\in nS$. 

The reduction number $r_J(I)$ can be computed by the following formula
\begin{align*}
r_J(I)&=\min\{n|\ I^{n+1}=JI^n\}\\
	  &=\min\{n|\ (n+1)E=(nE+\e_1)\cup(nE+\e_2)\}\\	
      &=\min\{n|\ (n+1)S=nS\cup (nS+d)\}.
\end{align*}

We now put all pieces together and present a concrete and effective algorithm to compare $\reg R(I)$ and $\reg F(I)$ for a monomial ideal $I$.

\begin{algorithm}
\caption{Comparing $\reg R(I)$ and $\reg F(I)$\\
\mbox{}\hspace{2.3cm} for $I=(x^d,x^{a_1}y^{d-a_1},\ldots,x^{a_p}y^{d-a_p},y^d)$}\label{compare}
\begin{algorithmic}[1]
\Procedure{EU-conjecture}{$I$}
   \State $S\gets \{0,a_1,\ldots,a_p,d\}$
   \State $T\gets \{0,d-a_p,\ldots,d-a_1,d\}$
   \State $r\gets 1$\Comment{\textit{Computing $r_J(I)$}}
   \While{$rS\cup (rS+d)\not=(r+1)S$}
      \State $r\gets r+1$
   \EndWhile
   \State $n\gets r$\Comment{\textit{Finding least $n\geq r$ such that $		
   							\big(\widetilde{I^n}\big)_{nd}= (I^{nd})_{nd}$}}
   \State $equal\gets false$
   \While{$equal=false$}
    	\For{$u\in \{1,\ldots,nd\}$}
    		\If{$(u\in \la S\ra)\wedge (nd-u\in \la T\ra)\wedge (u\not\in nS)$}
    			\State $equal\gets false$ \Comment{\textit{Found a monomial in $\big(\widetilde{I^n}
    															\big)_{nd}\setminus  (I^{nd})_{nd}$}}
				\State $n\gets n+1$    															
    			\Break
    		\Else
    			\State $equal\gets true$
    		\EndIf
    	\EndFor
    \EndWhile
	\For{$(u,v)\in \{1,\ldots,nd\}\times\{1,\ldots,nd\}$}\Comment{\textit{Start checking if $
	 																		\widetilde{I^n}=I^n$}}
   		\If{$(u\in \la S\ra)\wedge(v\in\la T\ra)$}\Comment{\textit{Pick a generator of $\widetilde{I^n}$}}
   			\State $belong\gets false$
			\For{$w\in \{nd-v,\ldots,u\}$}
				\If{$w\in nS$}		
   					\State $belong\gets true$\Comment{\textit{This generator is also in $I^n$}}
   					\Break
   				\EndIf
   			\EndFor
   			\If{$belong=false$}\Comment{\textit{Found a monomial in $\widetilde{I^n}\setminus I^n$}}
   				\Break
   			\EndIf
   		\EndIf
   \EndFor
   \If{$belong=false$}
   		\State \textbf{return FALSE}\Comment{\textit{The conjecture is false}}
   	\Else
   		\State \textbf{return TRUE}\Comment{\textit{The conjecture is true}}
   	\EndIf
\EndProcedure
\end{algorithmic}
\end{algorithm}

Since the conjecture of Eisenbud and Ulrich is true for monomial ideals generated by three elements \cite{RTT}, we ran the above algorithm for monomial ideals $I$ generated by four elements, i.e $I$ is of the form $(x^d,x^ay^{d-a},x^by^{d-b},y^d)$. We asked the algorithm to test for each $d$ all possible values of $a$ and $b$. The algorithm is implemented in C++ and is run on a personal computer with 4GB of RAM and an i5-7200U Processor. For each single case the algorithm returns the answer quite fast, but when $d$ increases the  number of possible combinations of $a$ and $b$ increases significantly. 

As a result, we found that the algorithm returns FALSE for the first time when $d=157$ and $a=35, b=98$, which means that the ideal $I=(x^{157},x^{35}y^{122},x^{98}y^{59}, y^{157})$ is a counter-example to the conjecture of Eisenbud and Ulrich.

\section{Analysis of the counter-example}

In this section we will present a sufficient condition for $\reg R(I) > \reg F(I)$ in the case $I$ is an equigenerated monomial ideal in $k[x,y]$. We then apply this result to give a theoretical proof for the above counter-example to the conjecture of Eisenbud and Ulrich. We shall also explore other invariants of $I$ with the help of Macaulay2.

We follow the same notation as in the previous section.

\begin{Theorem}\label{criterion} Let $I$ be a monomial ideal in $k[x,y]$ generated in degree $d$, and assume $I$ contains $x^d,y^d$. Let $n=\reg F(I)$. Suppose we can find positive integers $a, b$ such that $a+b=nd+1$ and one of the following two conditions is satisfied:

\noindent (i) $a-1\in\la S\ra, a\not\in\la S\ra$ and $b-1\in \la T\ra, b\not\in \la T\ra $, 

\noindent (ii) $a-1\not\in\la S\ra, a\in\la S\ra$ and $b-1\not\in \la T\ra, b\in \la T\ra.$

\noindent Then $\reg R(I) > \reg F(I)$.

\end{Theorem}

\begin{proof}
Let $a, b$ be positive integers such that $a+b=nd+1$. If either (i) or (ii) is satisfied, the monomials $x^ay^{b-1}$ and $x^{a-1}y^b$ are not in the minimal set of generators of $I^n$, thus $x^ay^b$ is not in $I^n$. On the other hand, since $nd-b=a-1$ and $nd-a=b-1$, by Theorem \ref{RR} the monomial $x^ay^b$ is in $\widetilde{I^n}$. This implies that $\widetilde{I^n}\neq I^n$. By Theorem \ref{RegRees1}, we have $\reg R(I) > n = \reg F(I)$.
\end{proof}

We shall also need the following result for the computation of $\reg F(I)$.

\begin{Lemma}\label{RegFiber2}
Let $I$ be a monomial ideal in $k[x,y]$ generated in degree $d$, and assume $I$ contains $x^d,y^d$. Let $J=(x^d,y^d)$. Then $\reg F(I)$ is the least integer $m\geq r_J(I)$ such that $(n+1)S\cap((n+1)S-d)=nS$ for all $n\geq m$. 
\end{Lemma}

\begin{proof}
By the proof of Theorem \ref{RR}, we have
$$
\big(\widetilde{I^n}\big)_{nd}=\{x^ry^s|\ (r,s)+t\e_1,(r,s)+t\e_2\in (n+t)E \text{ for some } t\geq 0\}.
$$
Thus by Theorem \ref{RegFiber1}, $\reg F(I)$ is the least integer $n\geq r_J(I)$ such that 
$$
\{\a\in\NN^2|\ \a+t\e_1,\a+t\e_2\in (n+t)E \text{ for some } t\geq 0\}=nE.
$$
Using the fact that for any $n\geq r_J(I)$ we have $(n+1)E=(nE+\e_1)\cup(nE+\e_2)$, we can easily see the following equivalence of equalities for some $m\geq r_J(I)$. 
\begin{align*}
\{\a\in\NN^2|\ \a+t\e_1,\a+t\e_2\in (m+t)E \text{ for some } t\geq 0\}&=mE \\
\Longleftrightarrow\{\a\in\NN^2|\ \a+t\e_1,\a+t\e_2\in (n+t)E \text{ for some } t\geq 0\}&=nE \text{ for all }n\geq m\\
\Longleftrightarrow\{\a\in\NN^2|\ \a+\e_1,\a+\e_2\in (n+1)E \}&=nE \text{ for all } n\geq m
\end{align*}
This follows that $\reg F(I)$ is the least integer $m\geq r_J(I)$  such that 
$$
((n+1)E-\e_1)\cap((n+1)E-\e_2)=nE \text{ for all } n\geq m.
$$ 
This equality implies the conclusion of the lemma.
\end{proof}

We now use the above results to analyze the example found in the previous section.

\begin{Proposition} \label{example}
Let $I=(x^{157},x^{35}y^{122},x^{98}y^{59},y^{157})$ and $J=(x^{157},y^{157})$ be ideals in $k[x,y]$. Then $\reg R(I) > \reg F(I)=r_J(I)=20$. 
\end{Proposition}

\begin{proof}
Let $S=\{0,35,98,157\}$ and $T=\{0,59,122,157\}$. We start with a calculation of the reduction number $r_J(I)$.

We will first show that the number $1141$ is in $20S\setminus (19S\cup (19S+157))$, which would then imply that $r_J(I)\geq 20$. The number $1141$ is in $20S$ because $1141=35\times 13+98\times 7$. Suppose 
$1141=35a+98b$ for some non-negative integers $a, b$ with $a+b\leq19$. This implies that $163=5a+14b$. We have $163\leq 5\times 19 + 9b$ which implies that $b\geq 8$. We also have $b<163/14$ which means that $b\leq 11$. We can check that none of these values for $b$ satisfies the equation. Now suppose that $1141=35a+98b+157c$ for some non-negative integers $a,b,c$ with $a+b+c\leq 20$ and $c\geq 1$. This implies that $c\equiv 0\mod 7$, hence $c=7$. We then have $42=35a+98b$ which clearly has no non-negative integer solutions.

We now prove that $21S\subseteq 20S\cup(20S+157)$. Let $p$ be an element in $21S$. We can write $p=35a+98b+157c$ for non-negative integers $a,b,c$ with $a+b+c\leq 21$. If $a+b+c\leq 20$ then $p$ is in $20S$. Suppose $a+b+c=21$. If $c\geq 1$ then $p$ is in $20S+157$. We can now let $c=0$, so $p=35a+98b$ with $a+b=21$.

If $0\leq a\leq 8$, we have $b\geq 13$, then
$$
p=35a+98b=35(a+5)+98(b-13)+157\times 7\in 20S+157.
$$
If $9\leq a\leq 13$, we have $b\geq 8$, then 
$$
p=35a+98b=35(a-9)+98(b-8)+157\times 7\in 20S+157.
$$
If $a\geq 14$ we have
$$p=35a+98b=35(a-14)+98(b+5)\in 20S.$$
This proves that $r_J(I)=20$.

We next prove that $\reg F(I)=20$. By Lemma \ref{RegFiber2} it suffices to show that $(n+1)S\cap((n+1)S-157)\subseteq nS$ for all $n\geq 20$. 

Let $p$ be an element in $(n+1)S\cap((n+1)S-157)$. There are non-negative integers $a,b,c$ with $a+b+c\leq n+1$ such that $p=35a+98b+157c$. If $a+b+c\leq n$ then $p$ is in $nS$, thus we can assume that $a+b+c=n+1$. Since $p+157$ is in $(n+1)S$, we can write
$$p+157=35a+98b+157(c+1)=35u+98v+157w,$$
for $u,v,w$ non-negative integers with $u+v+w\leq n+1$. If $w\geq 1$ then $p=35u+98v+157(w-1)$ which implies that $p$ is in $nS$. Thus we can assume that $w=0$, that is $p+157=35a+98b+157(c+1)=35u+98v$. Moreover, since $r_J(I) = 20$ we have $(n+1)S=nS\cup(nS+157)$ for all $n\geq 20$, thus we may assume that $u+v\leq 20$. 

From the equation $35a+98b+157(c+1)=35u+98v$ we have $c+1\equiv 0\mod 7$, thus $c=6$. Then $5a+14b+157=5u+14v$. Since $a+b=n+1-c\geq 15$, we have $9v=5a+14b+157-5(u+v)>5\times 15+157-5\times 20=132$, which implies that $v\geq 15$. Then 
$$p+157=35u+98v=35(u+5)+98(v-13)+157\times 7,$$ 
which implies that $p=35(u+5)+98(v-13)+157\times 6$, and it is an element of $nS$.

Finally, we will use Theorem \ref{criterion} to show that $\reg R(I)>20$. We shall prove that $1298$ is in $\la S\ra$ but $1299$ is not in $\la S \ra$, and $1841$ is in $\la T\ra$ but $1842$ is not in $\la T\ra$.

The fact that $1298$ is in $\la S\ra$ can be seen easily because $1298=13\times 35+7\times 98+1\times 157$. Suppose $1299=35a+98b+157c$ for some non-negative integers $a,b,c$. Then $3c\equiv 4\mod 7$, thus $c=6$. Then we have $51=5a+14b$ which clearly has no non-negative integer solutions.

Since $1841=12\times 59+8\times 122+1\times 157$, we have $1841$ is in $\la T\ra$. Suppose $1842=59a+122b+157c$ for some non-negative integers $a,b,c$. We see that $1842\geq 59(a+b+c)$, which implies that $a+b+c\leq 31$. Subtracting both sides of the equation from $157\times 31$ we obtain
$$3025=35u+98v+157w,$$
for non-negative integers $u,v,w$ with $u+v+w\leq 31$. Using modulo $7$ again we see that $3w\equiv 1\mod 7$, thus the possible values for $w$ are $5, 12$ and $19$.

If $w=19$ then $42=35u+98v$ which has no non-negative integer solutions. If $w=12$ then $1141=35u+98v$ with $u+v\leq 19$, which was shown to have no solution at the beginning of the proof. If $w=5$ then $320=5u+14v$ for $u+v\leq 26$. This implies that $320\leq 5\times 26+9v$, hence $v\geq 22$. On the other hand, $v\equiv 0 \mod 5$, so the only possible value for $v$ is $25$, but then $u$ would be negative. 
\end{proof}
\begin{Remark}\label{regRees} {\rm The regularity of the Rees algebra of the ideal $I$ in the counter-example is computed to be $21$, that is, the ideal $I^{21}$ is Ratliff-Rush closed. It would be interesting to explore the relationship between the regularity of the Rees algebra and of the fiber ring for general monomial ideals in $k[x,y]$.}
\end{Remark}
In the following we will use Macaulay2 to compute the relation type and the postulation number of the ideal $I$ in the counter-example. Those are invariants of $I$ that are known to have close relationship with the regularity of the associated graded ring $G(I)$, which equals the regularity of the Rees algeba $R(I)$ \cite{Ooi}. By definition, the {\it relation type} $\text{reltype} (I)$ of $I$ is the largest degree occurring in a minimal homogeneous system of generators of the defining ideal of the Rees algebra $R(I)$. Macaulay2 produces a minimal homogeneous system of generators of the defining ideal $J\subseteq k[x,y][w_0,w_1,w_2,w_3]$ of $R(I)$, among which the binomial generator   $w_1^9w_2^8-w_0^{10}w_3^7$ has maximal $w$-degree. Thus $\text{reltype} (I)=17$. Note that, in general the inequality $\text{reltype} (I)\leq \reg R(I) +1$ is always true \cite{Tr-1}.

For $I$ an $\mm$-primary ideal in the local ring $(A,\mm)$, the function 
$$H_I:\NN\rightarrow\NN,\,\, H_I(n)=\text{length}(A/I^{n}),$$
is called the Hilbert-Samuel function of $I$. The Hilbert-Samuel polynomial of $I$ is the polynomial $P_I(x)$ such that $H_I(n)=P_I(n)$ for $n\gg0$. The {\it postulation number} of $I$, denoted by $n(I)$, is the largest integer $n$ such that $H_I(n)\neq P_I(n)$. It is well-known, as a consequence of Grothendieck-Serre formula for the difference between the Hilbert function and Hilbert polynomial of a graded algebra, that $n(I)\leq \reg R(I)$. For the ideal $I$ in the counter-example, that bound yields $n(I)\leq 21$, by Remark ~\ref{regRees}. We then obtain the Hilbert-Samuel polynomial of $I$ as follows
$$P_I(n)=24649 {{n+1}\choose {2}}-12246n+11005.$$
A computation using Macaulay2 yields $H_I(20)=4942376\neq P_I(20)=4942375$, and $H_I(21)=5447758= P_I(21)$. Thus the postulation number of the ideal $I$ is 20. We note here that a bound for $n(I)$ not using the exact value of $\reg R(I)$ would be much larger. For example, we can use the bound given in \cite[Theorem 15]{Sch} to obtain that $n(I)\leq 2\times 157-4=310$. Computing the Hilbert-Samuel polynomial using this bound would be much more heavily.
\begin{Remark}\label{} {\rm Let $br(I)$ denote the big reduction number of an ideal $I$ which is defined by
$$
br(I) :=\max\{r_J(I)| \,\, J \text{ is a minimal reduction of } I\}.
$$
In \cite[Remark 2.7]{RTT}, the authors asked if $\reg R(I) = br(I)$ always holds in the context of Theorem \ref{RegRees1}. We show here that the ideal $I$ in the counter-example to the conjecture of Eisenbud and Ulrich provides a negative answer to this question as well. Indeed, let $\mathfrak{n}$ be the maximal graded ideal of the fiber ring $F(I)$. Since there is a natural correspondence between minimal reductions of $I$ and $\mathfrak{n}$, (see \cite[Lemma 3.1]{RTT}), we have $br(I)= r_{\mathfrak{q}}(\mathfrak{n})\leq \reg F(I) = 20 < \reg R(I)$, where $\mathfrak{q}$ is some minimal reduction of $\mathfrak{n}$.}
\end{Remark}

\noindent\textbf{Acknowledgment.} The author is grateful to Professor Ngo Viet Trung for many inspiring conversations, and to the referee for suggestions that help to improve the presentation of the paper.


\begin{thebibliography}{1}

\bibitem{CZ}
B. T. Cortadellas and S. Zarzuela,
On the structure of the fiber cone of ideals with analytic spread one,
J. Algebra 317 (2007), no. 2, 759--785. 

\bibitem{CHT} S. D. Cutkosky, J. Herzog and N. V. Trung, 
Asymptotic behaviour of the Castelnuovo-Mumford regularity,
Compositio Math. 118 (1999), no. 3, 243--261.

\bibitem{EU}
D. Eisenbud and B. Ulrich,
Notes on regularity stabilization, Proc. Amer. Math. Soc. 140 (2012), no. 4, 1221--1232.

\bibitem{El}
J. Elias, On the computation of Ratliff-Rush closure, 
J. Symbolic Comput. 37 (2004), 717--725.

\bibitem{M2} 
D. R. Grayson and M. E. Stillman,
Macaulay2, a software system  for research  in
algebraic geometry, available at
\url{http://www.math.uiuc.edu/Macaulay2/}

\bibitem{JN}
 A. V. Jayanthan and R. Nanduri,
Castelnuovo-Mumford Regularity and Gorensteiness of Fiber Cone, 
Comm. in Algebra 40 (2012), 1338--1351.

\bibitem{Ko} V. Kodiyalam, 
Asymptotic behaviour of Castelnuovo-Mumford regularity, 
Proceedings of Amer. Math. Soc. 128, no. 2, (1999), 407--411.

\bibitem{Li1}
C. H. Linh, 
Upper bound for the Castelnuovo-Mumford regularity of associated graded modules,
Comm. in Algebra 33 (2005), 1817--1831.

\bibitem{Ooi}
A. Ooishi, 
Genera and arithmetic genera of commutative rings, 
Hiroshima Math. J. 17 (1987), 47--66.

\bibitem{Cr}
V. C. Quinonez,
Ratliff-Rush monomial ideals, in: Algebraic and Geometric Combinatorics, Contemp. Math. 423 (2007), 43–-50.

\bibitem{RR} 
L. J. Ratliff, Jr. and D. E. Rush, 
Two notes on reductions of ideals, 
Indiana Univ. Math. J. 27 (1978), 929--934.

\bibitem{RS}
M. Rossi and I. Swanson,
Notes on the behavior of the Ratliff-Rush filtration, in: Commutative Algebra (Grenoble/Lyon, 2001),
Contemp. Math. 331 (2003), 313--328.

\bibitem{RTT}
M. Rossi, D. T. Trung and N. V. Trung, 
Castelnuovo-Mumford regularity and Ratliff-Rush closure, J. Algebra 504 (2018), 568--586.


\bibitem{RTV1}
M. Rossi, N. V. Trung and G. Valla, 
Castelnuovo-Mumford regularity and extended degree,
Trans. Amer. Math. Soc. 355 (2003), 1773-1786.

\bibitem{Sa}
J. Sally, 
Ideals whose Hilbert function and Hilbert polynomial agree at $n = 1$,
J. Algebra 157 (1993), 534--547.

\bibitem{Sch}
N. Schwartz,
Bounds for the postulation numbers of Hilbert functions,
J. Algebra 193 (1997), 581--615.

\bibitem{Tr-1}
N. V. Trung, 
Reduction exponent and degree bound for the defining equations of graded rings, 
Proc. Amer. Math. Soc. 101 (1987) 229--236.

\bibitem{Tr}
N. V. Trung,
The Castelnuovo regularity of the Rees algebra and the  associated graded ring,
Trans. Amer. Math. Soc. 35 (1998), 2813--2832.

\bibitem{Tr0}
N. V. Trung, 
Castelnuovo-Mumford regularity and related invariants, in: Commutative Algebra,
Lecture Notes Series 4, Ramanujan Mathematical Society, (2007), 157--180.

\bibitem{TW} 
N. V. Trung and H-S. Wang, 
On the asymptotic linearity of Castelnuovo-Mumford regularity, 
J. Pure Appl. Algebra 201 (2005), 42--48.

\end{thebibliography}
\end{document}